\documentclass[12pt,cd]{amsart}
\usepackage{latexsym}
\usepackage{amsmath,amscd,amssymb}
\usepackage{graphics}
 2

\usepackage{hyperref}
\newtheorem{theorem}{Theorem}[section]

\newtheorem{lemma}[theorem]{Lemma}

\newtheorem{remark}[theorem]{Remark}
\numberwithin{equation}{section}

\title{Weierstrass-Enneper representation for maximal surfaces in hodographic coordinates}
\author{Rahul Kumar Singh}
\address{Harish-Chandra Research Institue, HBNI, Chhatnag Road, Jhunsi, Allahabad-211019, India\\}
\email{rhlsngh498@gmail.com; rahulkumar@hri.res.in}
\begin{document}

\maketitle

\begin{abstract}
We obtain the Weierstrass-Enneper representation for maximal graphs (whose Gauss map is one-one) by directly computing the isothermal coordinates on it in Lorentz-Minkowski space. For this we use the method of Barbishov and Chernikov, which they used to find the solutions of Born-Infeld equation in hodographic coordinates. We could use their method in our case, because we realized that the maximal surface equation and Born-Infeld equation are related via a wick rotation in the first variable of the parametrising domain. 
\end{abstract}

\section{Introduction}
A maximal surface in Lorentz-Minkowski space $\mathbb{L}^3:=(\mathbb{R}^3, dx^2+dy^2-dz^2)$ is a spacelike surface whose mean curvature is zero everywhere. Any spacelike surface in $\mathbb{L}^3$ can be expressed locally as a graph $(x,y,\varphi(x,y))$ of some smooth function $\varphi$ which satisfies $\varphi_{x}^2+\varphi_{y}^2<1.$ Then any graph in $\mathbb{L}^3$ is maximal if $\varphi$ satisfies the following equation \cite{kobasingu}
\begin{equation}\label{maxse}
(1-\varphi_{x}^2)\varphi_{yy}+2\varphi_{x}\varphi_{y}\varphi_{xy}+(1-\varphi_{y}^2)\varphi_{xx}=0.
\end{equation} 
These surfaces, more generally, constant mean curvature surfaces are important in classical relativity\cite{marsden}. Similar to the case of minimal surfaces in $\mathbb{R}^3$ \cite{osser}, there exists a Weierstrass-Enneper formula for maximal surfaces in $\mathbb{L}^3$ \cite{kobamax}.
In fact, any maximal surface in $\mathbb{L}^3$ is represented as \cite{kobamax}
\begin{center}
$\Psi(\tau)=\Re\int(f(1+g^{2}), if(1-g^{2}), -2fg)d\tau, \tau\in D, D\subseteq \mathbb{C} $
\end{center}
$f$ is a holomorphic function on $ D $, $ g $ is a meromorphic function on $D$, $fg^{2}$ is holomorphic on $D$ and $\vert g(\tau)\vert\neq 1$ for $ \tau\in D $.\\ For a spacelike surface in $ \mathbb{L}^3 $, the Gauss map $G$ is defined as a map which assigns to a point of the surface $S$, the unit normal vector at that point. Therefore one can regard $ G:S\longrightarrow \mathbb{H}^2 $, $\mathbb{H}^2$ is a spacelike surface which has constant negative curvature $-1$ with respect to the induced metric.  \\
We can define a stereographic map $\sigma$ for $ \mathbb{H}^2 $ as
\begin{center}
$\sigma:\mathbb{C}\cup\{\infty\}-{\{\vert\tau\vert=1}\}\longrightarrow  \mathbb{H}^2$ by \\
\end{center}
 $$\sigma(\tau)=\left(\frac{-2\Re(\tau)}{\vert\tau\vert^2-1},\frac{-2\Im(\tau)}{\vert\tau\vert^2-1},\frac{\vert\tau\vert^2+1}{\vert\tau\vert^2-1}\right)$$ and $\sigma(\infty)=(0,0,1)$.

Since any maximal surface can be given isothermal coordinates one can think of $G$ as a map $G:D\subset \mathbb{C}\longrightarrow \mathbb{H}^2$, the Gauss map G is given by $G(\tau)=\sigma(g(\tau))$.\\
Next if we assume that the Gauss map for a maximal surface is one-one, then from above expression for $G$ we deduce that $g$ is one-one.
\\
Now set $ \zeta=g \Rightarrow\dfrac{d\zeta}{d\tau}=\dfrac{dg}{d\tau}\Rightarrow d\zeta=dg $.
Define $M(\zeta)=\dfrac{f}{\dfrac{dg}{d\tau}}=f\dfrac{d\tau}{dg}=f\dfrac{d\tau}{d\zeta}$.\\
Then $M(\zeta)d\zeta=fd\tau$. Hence in this new variable $\zeta$ we can write Weierstrass-Enneper representation for a maximal surface using just one meromorphic function $M(\zeta)$\cite{kobasingu}.
\begin{align}\label{we}
\Psi(\zeta)=\Re\int(M(\zeta)(1+\zeta^{2}), iM(\zeta)(1-\zeta^{2}), -2M(\zeta)\zeta)d\zeta
\end{align}

The induced metric can be given in terms of the meromorphic function $ M $ as \cite{kobamax}
\begin{align}\label{metric}
 ds^2=\left(\dfrac{\vert M(\zeta) \vert(1-\vert \zeta \vert^2)}{2}\right)^2|d\zeta|^2. 
\end{align}\\
\begin{remark}
It has been known that the Born-Infeld equation
\begin{align}\label{bie}
(1+\varphi_{x}^2)\varphi_{yy}-2\varphi_{x}\varphi_{y}\varphi_{xy}-(1-\varphi_{y}^2)\varphi_{xx}=0.
\end{align}
is related to the minimal surface equation
\begin{align}\label{minse}
(1+\varphi_{x}^2)\varphi_{yy}-2\varphi_{x}\varphi_{y}\varphi_{xy}+(1+\varphi_{y}^2)\varphi_{xx}=0.
\end{align}
via a wick rotation in second variable $ `y' $, i.e., if we replace $ y $ by $ iy $ in \eqref{bie} we get \eqref{minse}.\\
This fact has been exploited by Dey \cite{rukmini} in deriving Weirstrass- Enneper representation for minimal surfaces in Euclidean space, using Barbishov and Chernikov's technique\cite{barbi}. Whitham also explains this technique in his book \cite{whitham}.
\end{remark}
\begin{remark}
Here we observe that if instead of second variable, if we make a wick rotation in first variable $ `x' $, i.e., replacing $ x $ by $ ix $ in \eqref{bie} we get the maximal surface equation \eqref{maxse} and vice-versa. We use this fact to obtain the general solution to the maximal surface equation, namely its Weierstrass-Enneper representation.
\end{remark}

The paper is organised as follows: In section $ 2 $, we  derive the general solution of maximal surface equation by directly computing isothermal coordinates (namely $(\zeta_1, \zeta_2)$ where $ \zeta=\zeta_1+i\zeta_2 $) on it. As a consequence we get the single Weierstrass data which determines the maximal surface. In section $ 3 $, by using the Weierstrass data we write Weierstrass-Enneper representation in a new way. In section $ 4 $, we give some examples of maximal graphs in this new coordinate system illustrating the method we used to obtain the general solution.
Finally, in section $ 5 $ we construct a one parameter family of isometric maximal surfaces.

\section{Weierstrass-Enneper Representation}

We begin with the complex coordinates
\begin{align}\label{co}
\tilde{\xi}=i(x-iy)=i\bar{z} ~~~\text{,}~~~ \tilde{\eta}=i(x+iy)=iz ~~~\text{,}~~~
 \varphi_{i\bar{z}}=\tilde{u}=\frac{1}{i}\varphi_{\bar{z}} ~~~\text{,}~~~ \varphi_{iz}=\tilde{v}=\frac{1}{i}\varphi_z
\end{align}

\begin{center}
$\xi=x-iy=\bar{z} ~~~\text{,}~~~ \eta=x+iy=z ~~~\text{,}~~~
\varphi_{\bar{z}}=u ~~~\text{and}~~~ \varphi_{z}=v$.
\end{center}

The partial differentials in this new coordinates $ (\tilde{\xi},\tilde{\eta})$
is related to partial differentials in the old coordinates $(x,y)$ by the following relations
\begin{center}
$\varphi_{x}=i(\tilde{u}+\tilde{v}) ~~~\text{,}~~~ \varphi_{y}=\tilde{u}-\tilde{v}$~~~,~~~
$\varphi_{xx}=-(\tilde{u}_{\tilde{\xi}}+2\tilde{v}_{\tilde{\xi}}+\tilde{v}_{\tilde{\eta}})$~~~,~~~
$\varphi_{xy}=i(\tilde{u}_{\tilde{\xi}}-\tilde{v}_{\tilde{\eta}})  $~~~and~~~
\end{center}
$\varphi_{yy}=(\tilde{u}_{\tilde{\xi}}-2\tilde{v}_{\tilde{\xi}}+\tilde{v}_{\tilde{\eta}})$.\\
These identities reduces maximal surface equation \eqref{maxse} to
\begin{equation}\label{mse}
\tilde{v}^2\tilde{u}_{\tilde{\xi}}-(1+2\tilde{u}\tilde{v})\tilde{u}_{\tilde{\eta}}
+\tilde{u}^2\tilde{v}_{\tilde{\eta}}=0
~~~\text{and}~~~
\tilde{u}_{\tilde{\eta}}=\tilde{v}_{\tilde{\xi}}.
\end{equation}
Next we interchange the role of independent and dependent variables, i.e.,  \begin{center}$(\tilde{u},\tilde{v})\leftrightarrow(\tilde{\xi},\tilde{\eta})$.

\end{center}
Since we consider only those maximal graphs whose Gauss map is one-one, the Gaussian curvature $ K\neq0 $. Also $ 1-\varphi_x^2-\varphi_y^2\neq0 $ as $ \varphi_x^2+\varphi_y^2<1$. Therefore, 
\begin{center}
$ J=\tilde{u}_{\tilde{\xi}}\tilde{v}_{\tilde{\eta}}-\tilde{u}_{\tilde{\eta}}\tilde{v}_{\tilde{\xi}}=\dfrac{1}{4}(\varphi_{xy}^2-\varphi_{xx}\varphi_{yy})=\dfrac{K(1-\varphi_x^2-\varphi_y^2)^2}{4}\neq0 $.
\end{center}
Thus
\begin{center}
$ \tilde{v}_{\tilde{\eta}}=J\tilde{\xi}_{\tilde{u}} $~~~,~~~
$\tilde{v}_{\tilde{\xi}}=-J\tilde{\eta}_{\tilde{u}}$~~~,~~~
$ \tilde{u}_{\tilde{\eta}}=-J\tilde{\xi}_{\tilde{v}} $~~~,~~~
$\tilde{u}_{\tilde{\xi}}=J\tilde{\eta}_{\tilde{v}}$
\end{center}
 reduces equation \eqref{mse} to
\begin{equation}\label{mse1}
\tilde{v}^2\tilde{\eta}_{\tilde{v}}+(1+2\tilde{u}\tilde{v})\tilde{\xi}_{\tilde{v}}
+\tilde{u}^2\tilde{\xi}_{\tilde{u}}=0
~~~\text{and}~~~
\tilde{\xi}_{\tilde{v}}=\tilde{\eta}_{\tilde{u}}.
\end{equation}

Now if we use relations \eqref{co} in \eqref{mse1}, we get
\begin{equation}\label{mse3}
z_{u}-\bar{z}_{v}=0
~~~\text{and}~~~
v^2z_{v}-(1-2uv)z_{u}+u^2z_{v}=0
\end{equation}
or,
\begin{equation}\label{mse2}
\eta_{u}-\xi_{v}=0
~~~\text{and}~~~
v^2\eta_{v}-(1-2uv)\eta_{u}+u^2\eta_{v}=0.
\end{equation}
By differentiating equation \eqref{mse2} w.r.t $u$, we get a second order quasilinear pde
\begin{center}
$v^2\xi_{vv}-(1-2uv)\xi_{uv}+u^2\xi_{uu}=-2u\xi_{u}-2v\xi_{v}$
\end{center}
Now assuming that the solutions which we want to find is in hyperbolic regime, we find the characteristics for the above equation, they are integral curves of the following differential form
\begin{center}
$u^2dv^2+(1-2uv)dudv+v^2du^2=0$.
\end{center} 
 Characteristic curves are
 \begin{center}
$ \dfrac{1-\sqrt{1-4uv}}{2u}=c~~~,~~~\dfrac{1-\sqrt{1-4uv}}{2v}=c'$
 \end{center}
Now if we introduce
\begin{center}
$ \zeta=\dfrac{1-\sqrt{1-4uv}}{2v}~~~,~~~\bar{\zeta}=\dfrac{1-\sqrt{1-4uv}}{2u}$
\end{center} as new variables to replace $u$ and $v$, we get
\begin{align}\label{uv}
u=\dfrac{\zeta}{1+\zeta\bar{\zeta}}~~~,~~~v=\dfrac{\bar{\zeta}}{1+\zeta\bar{\zeta}}
\end{align}

\begin{lemma}
Equations \eqref{mse3} is equivalent to a single equation
\begin{center}\label{zeta}
$\zeta^2\bar{z}_{\zeta}-z_{\zeta}=0$.
\end{center}
\end{lemma} 
\begin{proof}
Since $u=\dfrac{\zeta}{1+\zeta\bar{\zeta}}~~~,~~~v=\dfrac{\bar{\zeta}}{1+\zeta\bar{\zeta}}$
we get \begin{center}
$ z_{\zeta}=\dfrac{z_{u}-\bar{\zeta}^2z_{v}}{(1+\zeta\bar{\zeta})^2} ~~~,~~~
 \bar{z}_{\zeta}=\dfrac{\bar{z}_{u}-\bar{\zeta}^2\bar{z}_{v}}{(1+\zeta\bar{\zeta})^2}$
   
\end{center}
Using above values of $ z_{\zeta}$ and $ \bar{z}_{\zeta}   $, we get
\begin{center}
 $\zeta^2\bar{z}_{\zeta}-z_{\zeta}=v^2z_{v}-(1-2uv)z_{u}+u^2z_{v}$,
 \end{center}
this shows \eqref{mse3} is equivalent to
\begin{center}
$\zeta^2\bar{z}_{\zeta}-z_{\zeta}=0$.
\end{center}

\end{proof}
\begin{theorem}\label{thm}
Any maximal surface whose Gauss map is one-one will have  a local Weierstrass-Enneper type representation of the following form
\begin{center}
$x(\zeta)=x_{0}+ \Re(\int^\zeta M(\omega)(1+\omega^{2})d\omega)$

\end{center}
\begin{center}
$ y(\zeta)=y_{0}+ \Re(\int^\zeta iM(\omega)(1-\omega^{2})d\omega)$
\end{center}
\begin{center}
$\varphi(\zeta)=\varphi_{0}+ \Re(\int^\zeta 2M(\omega)\omega d\omega)$
\end{center}
where $ M(\zeta) $ is,  a meromorphic function, known as the Weierstrass data.
\end{theorem}
\begin{remark}
Observe that $ \varphi\rightarrow -\varphi $ is a symmetry of the equation \eqref{maxse} if one keeps $ x $ and $ y $ invariant. Thus   $\varphi(\zeta)=\tilde{\varphi}_{0}+ \Re(\int^\zeta -2M(\omega)\omega d\omega)$ is also an acceptable representation.
\end{remark}

\begin{proof}
Since any maximal surface is locally a graph,
from above lemma, we have $\zeta^2\bar{z}_{\zeta}-z_{\zeta}=0$. By differentiating \eqref{zeta} with respect to $ \bar{\zeta} $, we obtain 
 \begin{equation}\label{zeta1}
 {\zeta}^2\bar{z}_{\zeta\bar{\zeta}}-z_{\zeta\bar{\zeta}}=0.
 \end{equation}
 Now use \eqref{zeta1} and its conjugate to obtain
 \begin{align*}
  \bar{z}_{\zeta\bar{\zeta}}=0 \Rightarrow \bar{z}=\bar{z_{0}}+F({\zeta})+H(\bar{\zeta}). 
 \end{align*}
 Then \begin{align}\label{zF}
 z=z_{0}+\overline{F({\zeta})}+\overline{H(\bar{\zeta})}.
 \end{align}
 
  Lemma also imply 
  $$\overline{H(\bar{\zeta})}=\int^{\zeta}\omega^2F'(\omega)d\omega$$.
  
  Thus, 
 \begin{equation*}
  \bar{z}=\bar{z_{0}}+F({\zeta})+\int^{{\zeta}}\bar{\omega}^2
  \overline{F'(\omega)}d\bar{\omega}.
  \end{equation*}
  Next we have 
 \begin{center}
  $\varphi_{\zeta}=\varphi_{i\bar{z}}(i\bar{z})_{\zeta}+\varphi_{iz}(iz)_{\zeta}=(u+v\zeta^2)F'(\zeta)=\zeta F'(\zeta)$.
  \end{center}
  Similarly,
  \begin{center}
  $\varphi_{\bar{\zeta}}=\bar{\zeta}\dfrac{d}{d\bar{\zeta}}(\overline{F({\zeta})})$
\end{center}
Hence 
\begin{equation}\label{phiF}
\varphi=\varphi_{0} + \int^ {\zeta} \omega F'(\omega) d\omega +\int^ {\bar{\zeta}} \bar{\omega}\frac{d}{d\bar{\omega}}(\overline{F({\omega})}) d\bar{\omega}.
\end{equation}
Let $F'(\omega)=M(\omega)$.
By expanding $z$ into its real and imaginary parts, also using $z+\bar{z}=2\Re(z)$ we get
\begin{center}
$x(\zeta)=x_{0}+\Re(\int^{\zeta} M(\omega)(1+\omega^{2})d\omega)$
\end{center}
\begin{center}
$ y(\zeta)=y_{0}+ \Re(\int^{\zeta} iM(\omega)(1-\omega^{2})d\omega)$
\end{center}
\begin{center}
$\varphi(\zeta)=\varphi_{0} + \Re(\int^{\zeta} 2M(\omega)\omega d\omega)$.
\end{center}

\end{proof}

\section{Hodographic coordinates}
If $ F'(\zeta)\neq 0 $. Then let $ H(\bar{\zeta})=\overline{F(\zeta)}=\bar{\rho} $ and $ F({\zeta})={\rho} $, so that we can regard $ \rho $ and $ \bar{\rho} $ as new variables, at least locally. Now in this new coordinate system $ \rho $, the Weiersrtass-Enneper representation attains the following form
\begin{equation}
x(\rho)=\frac{\rho+\bar{\rho}}{2}+\frac{1}{2}\left(\int(F^{-1}(\rho))^2d\rho+\int(H^{-1}(\bar{\rho}))^2d\bar{\rho}\right)
\end{equation}
\begin{equation}
y(\rho)=\frac{\bar{\rho}-\rho}{2i}+\frac{1}{2i}\left(\int(F^{-1}(\rho))^2d\rho-\int(H^{-1}(\bar{\rho}))^2d\bar{\rho}\right)
\end{equation}
\begin{equation}
\varphi(\rho)=\int F^{-1}(\rho)d\rho+\int H^{-1}(\bar{\rho})d\bar{\rho}
\end{equation}
From $(3.3)$ we have 
\begin{align}\label{hodo}
 \varphi_{\rho}=F^{-1}(\rho)=\zeta  ~~~\text{and}~~~ \varphi_{\bar{\rho}}=H^{-1}(\bar{\rho})=\bar{\zeta}
\end{align}
Now in terms of $\varphi_{\rho}$ and $\varphi_{\bar{\rho}}$ equations $(3.1)$ to $(3.3)$ reduces to
\begin{equation}
x(\rho)=\frac{\rho+\bar{\rho}}{2}+\frac{1}{2}\left(\int(\varphi_{\rho})^2d\rho+\int(\varphi_{\bar{\rho}})^2d\bar{\rho}\right)
\end{equation}
\begin{equation}
y(\rho)=\frac{\bar{\rho}-\rho}{2i}+\frac{1}{2i}\left(\int(\varphi_{\rho})^2d\rho-\int(\varphi_{\bar{\rho}})^2d\bar{\rho}\right)
\end{equation}
\begin{equation}
\varphi(\rho)=\varphi(\rho)+\varphi(\bar{\rho})
\end{equation}
Now if we write $ \rho=\rho_{1}+i\rho_{2} $. Then one can easily check that $\rho_{1}$ and $ \rho_{2} $ are isothermal, i.e., 
\begin{center}
$ \vert X_{\rho_{1}} \vert_{L}=\vert X_{\rho_{2}} \vert_{L} $ and $ \langle X_{\rho_{1}}, X_{\rho_{2}}\rangle_{L}=0$
\end{center}
where $X=(x,y,\varphi)$ and $\langle,\rangle_{L}$ is the Lorentzian norm.
 Since the coordinate system $ \rho $ is related to the coordinate system $ \zeta $ by a holomorphic map F, the coordinate system $ \zeta=\zeta_1+i\zeta_2 $ is also isothermal.
Also the expression for unit normal to the maximal surface depends only on $ \varphi_{\rho} $, as we have

$$ N=\dfrac{X_{\rho_{1}} \times_{L} X_{\rho_{2}}}{\vert X_{\rho_{1}} \times_{L} X_{\rho_{2}} \vert_{L}}=\left(\frac{2\Re(\varphi_{\rho})}{1-\vert\varphi_{\rho}\vert^{2}}, \frac{2\Im(\varphi_{\rho})}{1-\vert\varphi_{\rho}\vert^{2}},-\frac{1+\vert\varphi_{\rho}\vert^{2}}{1-\vert\varphi_{\rho}\vert^{2}}\right). $$

So geometrically $ \varphi_{\rho} $ represents  the stereographic projection of the Gauss map.

\section{Examples} 

\textit{Lorentzian Catenoid:}\cite{kobamax} Consider 
\begin{align}\label{catenoid} 
\varphi(x,y)=\sinh^{-1}(\sqrt{x^2+y^2})=\sinh^{-1}(\sqrt{z\bar{z}}) 
\end{align} which is a maximal surface in the Lorentz-Minkowski space whose Gauss map is one-one. Then
\begin{align*}
\varphi_z=\dfrac{\bar{z}}{2|z|\sqrt{|z|^2+1}} ~~~\text{,}~~~ \varphi_{\bar{z}}=\dfrac{z}{2|z|\sqrt{|z|^2+1}} 
\end{align*} 
Recall \eqref{co} and \eqref{uv}, we get
\begin{align}\label{zz}
\frac{u}{v}=\frac{z}{\bar{z}}=\frac{\zeta}{\bar{\zeta}}.
\end{align}
Next we have
\begin{align}\label{e1}
\dfrac{\zeta}{1+\zeta\bar{\zeta}}=\dfrac{z}{2|z|\sqrt{|z|^2+1}},
\end{align}
we use this equation to obtain $z$ in terms of $ \zeta $ and then from this we get to know the single holomorphic function $ F(\zeta) $. Squaring both the sides of equation \eqref{e1} and using the relations \eqref{zz} in between, we get
\begin{align*}
z^2=\left(\dfrac{1}{2}\left(\dfrac{1}{\bar{\zeta}}-\zeta\right)\right)^2
\end{align*}
taking positive square root
\begin{align*}
z=\dfrac{1}{2}\left(\dfrac{1}{\bar{\zeta}}-\zeta\right).
\end{align*} 

Comparing this with \eqref{zF}, we obtain $ \overline{F({\zeta})}=\dfrac{1}{2\bar{\zeta}} $, so 
  we have $ F(\zeta)=\dfrac{1}{2\zeta}.$ Therefore we can compute the Weierstrass data as $ M(\zeta)=F'(\zeta)=\frac{-1}{2\zeta^2}$.
Now $ \varphi(\zeta,\bar{\zeta}) $ can be computed by the formula \eqref{phiF}. Infact
\begin{align}\label{cate}
 \varphi(\zeta,\bar{\zeta})=-\frac{1}{2}\log(\zeta\bar{\zeta}) 
\end{align}
$$x= -\frac{1}{2} Re\left(\zeta-\frac{1}{\zeta}\right);\; \;\;\; y= -\frac{1}{2} Im\left(\zeta+\frac{1}{\zeta}\right).$$
This is Weierstrass-Enneper representation for maximal graph Lorentzian catenoid in terms of the coordinates $(\zeta, \overline{\zeta})$.
Next we write $(x, y, \varphi) $ in terms of hodographic coordinates $ (\rho, \overline{\rho}) $ \eqref{hodo}.
\begin{align}\label{hodophi}
\varphi(\rho,\bar{\rho})=\frac{1}{2}(\log(2\rho)+\log(2\bar{\rho}))
\end{align}
$$x= -\frac{1}{2} Re\left(\frac{1}{2\rho}-2\rho\right);\; \;\;\; y= -\frac{1}{2} Im\left(\frac{1}{2\rho}+2\rho\right).$$
\textit{Lorentzian Helicoid}\cite{kobamax}\footnote{Plane and Helicoid are the only maximal surfaces in Lorentz-Minkowski space which are also minimal surfaces in Euclidean space.}: Consider 
\begin{align}
\varphi(x,y)=\frac{\pi}{2}+\tan^{-1}\left(\frac{y}{x}\right)=\frac{\pi}{2}+\tan^{-1}\left(\frac{1}{i}\left(\frac{z-\bar{z}}{z+\bar{z}})\right)\right).
\end{align}
Then 
\begin{align}\label{iz}
u=\varphi_{\bar{z}}=\frac{i}{2\bar{z}}  ~~~\text{and}~~~ v=\varphi_{z}=\frac{-i}{2z} 
\end{align}
Again we have
\begin{align}
\frac{u}{v}=\frac{{-z}}{\bar{z}}=\frac{\zeta}{\bar{\zeta}}.
\end{align}
Using relations \eqref{uv},  \eqref{zF} and \eqref{iz}, we found 
\begin{align}
\overline{F({\zeta})}=\frac{-i}{2\bar{\zeta}} ~~~\text{and}~~~ F(\zeta)=\frac{i}{2\zeta}
\end{align}
and hence Weierstrass data $ M(\zeta)=F'(\zeta)=\frac{-i}{2\zeta^2} $.
Therefore 
\begin{align}
\varphi(\zeta,\bar{\zeta})=-\frac{i}{2}\log\left(\frac{{\zeta}}{\overline{\zeta}}\right)
\end{align}
$$x= \frac{1}{2}Im\left(\zeta-\frac{1}{\zeta}\right);\;\;\; y= -\frac{1}{2}Re\left(\zeta+\frac{1}{\zeta}\right).$$

Now again we write $(x, y, \phi)$ in terms of hodographic coordinates $\rho$ and $\bar{\rho}$ as
$$x= \frac{1}{2} Im \left(\frac{i}{2\rho}-\frac{2\rho}{i}\right);\;\; y=- \frac{1}{2}Re\left(\frac{i}{2\rho}+\frac{2\rho}{i}\right)$$
$$\phi= \frac{-i}{2}\log\left(\frac{-\bar{\rho}}{\rho}\right).$$
\section{One parameter family of isometric maximal surfaces}

In the previous section of examples we also computed the Weierstrass data $ M_{c}(\zeta)=\dfrac{-1}{2\zeta^2} $ for Lorentzian catenoid and $ M_{h}(\zeta)=\dfrac{-i}{2\zeta^2} $ for Lorentzian helicoid. Now if we define $ \forall$ $ \theta , $ such that $ 0\leq\theta\leq\frac{\pi}{2} $
\begin{align}
M_{\theta}(\zeta)=e^{i\theta}M(\zeta)~~~\text{, where}~~~ M(\zeta)=\dfrac{-1}{2\zeta^2}
\end{align}
then $ M_{\theta}(\zeta) $ becomes the Weierstrass data for a maximal surface which can be obtained using Theorem \ref{thm}. 
In particular, when $ \theta=0 $,$ M_{0}(\zeta)=M_{c}(\zeta) $, we get back Lorentzian catenoid and when $ \theta=\frac{\pi}{2} $, $ M_{\frac{\pi}{2}}(\zeta)=M_{h}(\zeta) $, we get the Lorentzian helicoid. Next recall the expression for the metric \eqref{metric},  here we see that the metric depends only on the modulus of Weierstrass data $ M(\zeta) $, so if we replace $ M(\zeta) $ by $ e^{i\theta}M(\zeta) $ in the expression of the metric, the form of the metric remains unchanged because $ |M(\zeta)|=|e^{i\theta}M(\zeta)|. $ This tells us that by varying $ \theta $, we get a one parameter family of isometric maximal surfaces in Lorentz-Minkowski space.\\
In general, if one starts with a Weierstrass data for a  given maximal surface, one can construct a one parameter family of isometric maximal surfaces, by following the procedure described in previous paragraph, starting from the given surface.\\

\textbf{Acknowledgement}\\

I would like to thank my advisor Dr. Rukmini Dey for useful discussions.

\end{document}